\documentclass[a4paper]{amsart}

\usepackage{amssymb,amsmath,amsthm,amscd,amsfonts,bbm,mathrsfs}

\usepackage[cmtip,all]{xy}

\newtheorem{lemma}{Lemma}
\newtheorem{thm}[lemma]{Theorem}
\newtheorem{prop}[lemma]{Proposition}
\theoremstyle{definition}

\newtheorem{rmk}[lemma]{Remark}

\newcommand{\AAA}{\mathcal{A}}
\newcommand{\BBB}{\mathscr{B}}

\newcommand{\DDD}{\mathscr{D}}
\newcommand{\EEE}{\mathcal{E}}
\newcommand{\FFF}{\mathcal{F}}
\newcommand{\GGG}{\mathcal{G}}
\newcommand{\JJJ}{\mathcal{J}}

\newcommand{\MMM}{\mathcal{M}}

\newcommand{\OOO}{\mathcal{O}}
\newcommand{\PPP}{\mathscr{P}}
\newcommand{\RRR}{\mathcal{R}}

\newcommand{\Fm}{\mathfrak{m}}

\newcommand{\FS}{\mathfrak{S}}

\newcommand{\DD}{{\mathbb{D}}}
\newcommand{\EE}{{\mathbb{E}}}
\newcommand{\FF}{{\mathbb{F}}}

\newcommand{\II}{{\mathbb{I}}}

\newcommand{\QQ}{{\mathbb{Q}}}
\newcommand{\WW}{{\mathbb{W}}}
\newcommand{\ZZ}{{\mathbb{Z}}}

\newcommand{\cc}{{\mathbbm{c}}}
\newcommand{\ff}{{\mathbbm{f}}}
\newcommand{\vv}{{\mathbbm{v}}}
\newcommand{\uu}{{\mathbbm{u}}}

\DeclareMathOperator{\BT}{BT}

\DeclareMathOperator{\Ext}{Ext}
\DeclareMathOperator{\Fil}{Fil}
\DeclareMathOperator{\Frac}{Frac}

\DeclareMathOperator{\Hom}{Hom}
\DeclareMathOperator{\Gal}{Gal}
\DeclareMathOperator{\Ker}{Ker}

\DeclareMathOperator{\Spec}{Spec}

\DeclareMathOperator{\Tor}{Tor}

\DeclareMathOperator{\cris}{cris}

\DeclareMathOperator{\id}{id}

\DeclareMathOperator{\nr}{nr}

\DeclareMathOperator{\per}{per}

\DeclareMathOperator{\sep}{sep}

\begin{document}

\title{Displayed equations for Galois representations}
\author{Eike Lau}
\address{Fakult\"{a}t f\"{u}r Mathematik,
Universit\"{a}t Bielefeld, D-33501 Bielefeld}
\email{lau@math.uni-bielefeld.de}

\date{\today}

\begin{abstract} 
The Galois representation associated to a
$p$-divisible group over a 
complete noetherian normal local ring
with perfect residue field is described in terms of its 
Dieudonn\'e display. 
As a corollary we deduce in arbitrary characteristic
Kisin's description of the Galois representation 
associated to a commutative finite flat 
$p$-group scheme over a $p$-adic discrete valuation ring
in terms of its Breuil-Kisin module.
This was obtained earlier by W.\ Kim by a different method.
\end{abstract}

\maketitle

\section*{Introduction}

\renewcommand{\thelemma}{\Alph{lemma}}

Let $R$ be a complete noetherian normal local ring with
perfect residue field $k$ of positive characteristic $p$
and with fraction field $K$ of characteristic zero.
For a $p$-divisible group $G$ over $R$, the Tate module
$T_p(G)$ is a free $\ZZ_p$-module of finite rank with a 
continuous action of the absolute Galois group $\GGG_K$. 
We want to describe the Tate module 
in terms of the Dieudonn\'e display
$\PPP=(P,Q,F,F_1)$ associated to $G$ in \cite{Zink-DDisp}
and \cite{Lau-Relation}, and relate this to other
descriptions of the Tate module 
when $R$ is a discrete valuation ring.

Let us recall that the Zink ring $\WW(R)$ is a
subring of the ring of Witt vectors $W(R)$ which is
stable under the Frobenius endomorphism $f$ of $W(R)$.
Let $v$ be the Verschiebung of $W(R)$ and let
$u_0\in W(R)$ be the unit defined by $u_0=1$ 
if $p$ is odd and by $v(u_0)=2-[2]$ if $p=2$.
The components of $\PPP$ as above are a finite free 
$\WW(R)$-module $P$,
a submodule $Q$ of $P$ such that $P/Q$ is a free $R$-module,
and $f$-linear maps $F:P\to P$ and $F_1:Q\to P$ 
such that the image of $F_1$ generates $P$, and 
$F_1(v(u_0a)x)=aF(x)$ for $x\in P$ and $a\in\WW(R)$.
The twist by $u_0$ is necessary 
since $v$ does not stabilise $\WW(R)$ when $p=2$.


Let $\hat R^{\nr}$ be the completion of the strict henselisation of $R$,
let $\tilde K$ be an algebraic closure of its fraction field $\hat K^{\nr}$, 
let $\tilde R\subset\tilde K$ be the integral closure of $\hat R^{\nr}$, 
and let 
\[
\WW(\tilde R)=\varinjlim_E\WW(R_E)
\]
where $E$ runs through the finite extensions of $\hat K^{\nr}$ contained in $\tilde K$ and where $R_E\subset E$ is the integral closure of $\hat R^{\nr}$.
Let $\hat{\tilde R}$ be the $p$-adic completion of $\tilde R$ and
$\hat\WW(\tilde R)$ the $p$-adic completion of $\WW(\tilde R)$.
We define
\[
\hat P_{\tilde R}=\hat\WW(\tilde R)\otimes_{\WW(R)}P
\]
and
\[
\hat Q_{\tilde R}=\Ker(\hat P_{\tilde R}\to\hat{\tilde R}\otimes_RP/Q).
\]
Let $\bar K\subset\tilde K$ 
be the algebraic closure of $K$
and let $\tilde\GGG_K$ be the group
of automorphisms of $\tilde K$ whose restriction to 
$\bar K\hat K^{\nr}$ is induced by an element of $\GGG_K$.
The natural map $\tilde\GGG_K\to\GGG_K$ is surjective,
and bijective when $R$ is one-dimensional since then
$\tilde K=\bar K\hat K^{\nr}$.
The following is the main result of this note; see Proposition \ref{Pr-Tate-perfect}.

\begin{thm}
\label{Th-A}
There is an exact sequence of $\tilde\GGG_K$-modules
\[
0\longrightarrow T_p(G)\longrightarrow\hat Q_{\tilde R}\xrightarrow{F_1-1}
\hat P_{\tilde R}\longrightarrow 0.
\]
\end{thm}

If $G$ is connected, a similar description of $T_p(G)$ 
in terms of the nilpotent display of $G$ is part of
Zink's theory of displays.
In this case $k$ need not be perfect; 
see \cite[Proposition 4.4]{Messing-Disp}. The
proof is recalled in Proposition \ref{Pr-Tate-connected} below.

\medskip

Assume now in addition that $R$ is a discrete valuation ring.
Then Theorem \ref{Th-A}
can be related with the descriptions of $T_p(G)$
in terms of $p$-adic Hodge theory and
in terms of Breuil-Kisin modules as follows.

\subsubsection*{Relation with the crystalline period homomorphism}

Let $M_{\cris}$ be the value of the covariant
Dieudonn\'e crystal of $G$ over $A_{\cris}(R)$.
It carries a filtration and a Frobenius, 
and by \cite{Faltings-Integral} there is a period homomorphism
\[
T_p(G)\to\Fil M_{\cris}^{F=p}
\] 
which is bijective if $p$ is odd, and injective with
cokernel annihilated by $p$ if $p=2$.
The $v$-stabilised Zink ring $\WW^+(R)=\WW(R)[v(1)]$
induces an extension $\hat\WW^+(\tilde R)$ of the ring 
$\hat\WW(\tilde R)$ defined above; the extension is
trivial if $p$ is odd. 
Since the $v$-stabilised Zink ring carries divided powers,
the universal property of $A_{\cris}$ gives a homomorphism 
\[
\varkappa_{\cris}:A_{\cris}(R)\to\hat\WW^+(\tilde R).
\]
Using the crystalline description of Dieudonn\'e displays of  \cite{Lau-Relation}, 
one obtains an $A_{\cris}$-linear map
\[
M_{\cris}\xrightarrow\tau\hat\WW^+(\tilde R)
\otimes_{\hat\WW(\tilde R)}\hat P_{\tilde R}
\]
compatible with Frobenius and filtration.
We will show that $\tau$ induces
the identity on $T_p(G)$, viewed as a
submodule of $\Fil M_{\cris}$ by the period homomorphism
and as a submodule of $
\hat Q_{\tilde R}\subset\hat P_{\tilde R}$ by Theorem \ref{Th-A};
see Proposition \ref{Pr-commute}.

\subsubsection*{Relation with Breuil-Kisin modules}

Choose a generator $\pi$ of the maximal ideal of $R$.
Let $\FS=W(k)[[t]]$ and let $\sigma:\FS\to\FS$
extend the Frobenius automorphism of $W(k)$ by $t\mapsto t^p$;
see below for the case of more general Frobenius lifts.
We consider pairs $M=(M,\phi)$ where $M$ is a finite $\FS$-module
and where $\phi:M\to M^{(\sigma)}$ is an $\FS$-linear map 
with cokernel annihilated by the kernel of the map
$\FS\to R$ given by $t\mapsto\pi$.
Following \cite{Vasiu-Zink}, $M$ is called a Breuil window 
if $M$ is free over $\FS$, and $M$ is called a Breuil
module if $M$ is a $p$-torsion $\FS$-module of projective 
dimension at most one.

It is known that $p$-divisible groups over $R$ are
equivalent to Breuil windows. This was conjectured
by Breuil \cite{Breuil} and proved by Kisin 
\cite{Kisin-crys,Kisin-2adic} if $p$ is odd, 
and for connected groups if $p=2$.
The general case is proved in \cite{Lau-Relation} by showing 
that Breuil windows are equivalent to Dieudonn\'e displays;
here the ring $R$ can be regular of arbitrary dimension.
For odd $p$ the last equivalence is already proved in
\cite{Vasiu-Zink} for some regular rings, including all
discrete valuation rings.
As a corollary, commutative finite flat $p$-group schemes
over $R$ are equivalent to Breuil modules.
Another proof for $p=2$, related more closely to Kisin's methods,
was obtained independently by W.~Kim \cite{Kim}.

Let $K_\infty$ be the extension of $K$ generated by
a chosen system of successive $p$-th roots of $\pi$.
For a $p$-divisible group $G$ over $R$ let $T(G)$ be its Tate module,
and for a commutative finite flat $p$-group
scheme $G$ over $R$ let $T(G)=G(\bar K)$. 
Kisin's and Kim's results include a description of $T(G)$ 
as a $\GGG_{K_\infty}$-representation in terms of the 
Breuil window or Breuil module $(M,\phi)$ associated to $G$.
In the covariant theory it takes the following form:
\begin{equation}
\label{Eq-second}
T(G)=\{x\in M^{\nr}\mid\phi(x)=1\otimes x\text{ in }
\FS^{\nr}\otimes_{\sigma,\FS^{\nr}}M^{\nr}\}
\end{equation}
Here $M^{\nr}=\FS^{\nr}\otimes_\FS M$, and the ring $\FS^{\nr}$ 
is recalled in Section \ref{Se-FS-nr} below.

We will show how \eqref{Eq-second} 
can be deduced from Theorem \ref{Th-A}.
It suffices to consider the case where $G$ is a $p$-divisible group.
The equivalence between Breuil windows and Dieudonn\'e 
displays over $R$ is induced by a homomorphism 
$\varkappa:\FS\to\WW(R)$. It can be extended to
\[
\varkappa^{\nr}:\FS^{\nr}\to\hat\WW(\tilde R),
\]
which allows to define a map of $\GGG_{K_\infty}$-modules
\[
\{x\in M^{\nr}\mid \phi(x)=1\otimes x\}
\xrightarrow\tau
\{x\in\hat Q_{\tilde R}\mid F_1(x)=x\}.
\]
Since the target is isomorphic to $T(G)$ by Theorem \ref{Th-A},
the proof of \eqref{Eq-second} is reduced to showing that
$\tau$ is bijective; see Proposition \ref{Pr-bijective}.
The verification is easy if $G$ is \'etale; the general 
case follows quite formally using a duality argument.

\medskip

Finally let us recall that the equivalence between
Breuil windows and $p$-divisible groups requires only 
a Frobenius lift $\sigma:\FS\to\FS$ which stabilises the ideal
$t\FS$ such that $p^2$ divides the linear term 
of the power series $\sigma(t)$. 
In this case, if $K_\infty$ denotes the extension of $K$ generated 
by a chosen system of successive $\sigma(t)$-roots of $\pi$,
we obtain an isomorphism \eqref{Eq-second} of 
$\GGG_{K_\infty}$-modules as before;
here the ring $\FS^{\nr}$ depends on $\sigma$ as well.

\medskip

The author thanks Th.~Zink for interesting and helpful discussions.


\numberwithin{lemma}{section}
\numberwithin{equation}{section}

\section{The case of connected $p$-divisible groups}
\label{Se-Connected}

Let $R$ be a complete noetherian normal local ring with residue
field $k$ of characteristic $p\ne 0$, with fraction field $K$ of
characteristic zero, and with maximal ideal $\Fm$.
In this section we recall how the Tate module of a connected 
$p$-divisible group over $R$ is expressed in terms of its
nilpotent display.

Fix an algebraic closure $\bar K$ of $K$ 
and let $\GGG_K=\Gal(\bar K/K)$.
Let $\bar R\subset\bar K$
be the integral closure of $R$ and let $\bar\Fm\subset\bar R$ 
be its maximal ideal.
For a finite extension $E$ of $K$ contained in $\bar K$
let $R_E=\bar R\cap E$,
which is a complete noetherian local ring, and let
$\Fm_E\subset R_E$ be its maximal ideal. We write
\[
\hat W(\Fm_E)=\varprojlim_n\hat W(\Fm_E/\Fm_E^n);
\qquad
\hat W(\bar\Fm)=\varinjlim_E\hat W(\Fm_E).
\]
Let $\bar W(\bar\Fm)$ be the $p$-adic completion of 
$\hat W(\bar\Fm)$
and let $\hat{\bar\Fm}$ be the $p$-adic completion of $\bar\Fm$.
We have a surjection $\bar W(\bar\Fm)\to\hat{\bar\Fm}$.
For a display $\PPP=(P,Q,F,F_1)$ over $R$ let
\[
\bar P_{\bar\Fm}=\bar W(\bar\Fm)\otimes_{W(R)}P;
\qquad
\bar Q_{\bar\Fm}=\Ker(\bar P_{\bar\Fm}\to\hat{\bar\Fm}\otimes_RP/Q).
\]
The functor $\BT$ of \cite{Zink-Disp} induces 
an equivalence of categories between nilpotent displays 
over $R$ and connected $p$-divisible groups over $R$; 
here $\PPP$ is called nilpotent if $\PPP\otimes_Rk$ is
$V$-nilpotent in the usual sense. 
The following is stated in \cite[Proposition~4.4]{Messing-Disp}.

\begin{prop}[Zink]
\label{Pr-Tate-connected}
Let $\PPP$ be a nilpotent display over $R$ and let 
$G=\BT(\PPP)$ be the associated connected $p$-divisible group over $R$.
There is a natural exact sequence of $\GGG_K$-modules
\[
0\longrightarrow T_p(G)\longrightarrow\bar Q_{\bar\Fm}\xrightarrow{F_1-1}
\bar P_{\bar\Fm}\longrightarrow 0.
\]
\end{prop}

Here $T_p(G)=\Hom(\QQ_p/\ZZ_p,G(\bar K))$ is the Tate module of $G$.

The proof of Proposition \ref{Pr-Tate-connected}
uses the following well-known facts.

\begin{lemma}
\label{Le-Ext1}
Let $A$ be an abelian group.
\begin{enumerate}
\renewcommand{\theenumi}{\roman{enumi}}
\item 
If $A$ has no $p$-torsion then 
$\Ext^1(\QQ_p/\ZZ_p,A)=\varprojlim A/p^nA$.
\item
If $pA=A$ then $\Ext^1(\QQ_p/\ZZ_p,A)$ is zero.
\end{enumerate}
\end{lemma}

\begin{proof}
The group $\Hom(\QQ_p/\ZZ_p,A)$ is isomorphic to
$\varprojlim\Hom(\ZZ/p^n\ZZ,A)$ with transition maps induced
by $p:\ZZ/p^n\ZZ\to\ZZ/p^{n+1}\ZZ$. The corresponding system
$\Ext^1(\ZZ/p^n\ZZ,A)$ is isomorphic to $A/p^nA$ with
transition maps induced by $\id_A$. Thus there is an
exact sequence 
\[
0\to{\varprojlim}^1A[p^n]\to\Ext^1(\QQ_p/\ZZ_p,A)
\to\varprojlim A/p^nA\to 0.
\]
Both assertions of the lemma follow easily.
\end{proof}

For a $p$-divisible group $G$ over $R$ and 
for $E$ as above we write
\[
\hat G(R_E)=\varprojlim_n G(R_E/\Fm_E^n);
\qquad
\hat G(\bar R)=\varinjlim_E\hat G(R_E).
\]

\begin{lemma}
\label{Le-p-surj}
Multiplication by $p$ is surjective on $\hat G(\bar R)$.
\end{lemma}

\begin{proof}
Let $x\in\hat G(R_E)$ be given.
The inverse image of $x$ under multiplication by $p$ is a
compatible system of $G[p]$-torsors $Y_n$ over
$R_E/\Fm_E^n$. They define a $G[p]$-torsor $Y$
over $R_E$. For some finite extension $F$ of
$E$ the set $Y(F)=Y(R_F)$ is non-empty, and
$x$ becomes divisible by $p$ in $\hat G(R_F)$.
\end{proof}

\begin{proof}[Proof of Proposition \ref{Pr-Tate-connected}]
Let $E$ be a finite Galois extension of $K$
in $\bar K$. Let
\[
\hat P_{E,n}=\hat W(\Fm_E/\Fm_E^n)\otimes_{W(R)}P;
\quad
\hat Q_{E,n}=\Ker(\hat P_{E,n}\to\Fm_E/\Fm_E^n\otimes_R P/Q).
\]
Recall that $P$ is a finite free $W(R)$-module, 
and $P/Q$ is a finite free $R$-module. 
The definition of the functor $\BT$ in \cite[Thm.~81]{Zink-Disp}
gives an exact sequence
of $\GGG_K$-modules
\begin{equation*}
0\longrightarrow\hat Q_{E,n}\xrightarrow{F_1-1}
\hat P_{E,n}\longrightarrow G(R_E/\Fm_E^n)\longrightarrow 0.
\end{equation*}
Since the modules $\hat Q_{E,n}$ form a surjective system 
with respect to $n$, applying $\varinjlim_E\varprojlim_n$
gives an exact sequence of $\GGG_K$-modules
\begin{equation}
\label{Eq-BT}
0\longrightarrow\hat Q_{\bar\Fm}\xrightarrow{F_1-1}
\hat P_{\bar\Fm}\longrightarrow\hat G(\bar R)\longrightarrow 0
\end{equation}
with $\hat P_{\bar\Fm}=\hat W(\bar\Fm)\otimes_{W(R)}P$ and
$\hat Q_{\bar\Fm}=\Ker(\hat P_{\bar\Fm}\to\bar\Fm\otimes_RP/Q)$.
The $p$-adic completions of $\hat P_{\bar\Fm}$ and
$\hat Q_{\bar\Fm}$ are $\bar P_{\bar \Fm}$ and $\bar Q_{\bar\Fm}$;
here we use that $\bar\Fm\otimes_RP/Q$ has no $p$-torsion.
Moreover $\hat P_{\bar\Fm}$ has no $p$-torsion
since $\hat W(\bar\Fm)$ is contained in the
$\QQ$-algebra $W(\bar K)$.
Using Lemmas \ref{Le-p-surj} and \ref{Le-Ext1}, the
$\Ext$-sequence of $\QQ_p/\ZZ_p$ with \eqref{Eq-BT} 
reduces to the short exact sequence
\[
0\longrightarrow\Hom(\QQ_p/\ZZ_p,\hat G(\bar R))\longrightarrow\bar Q_{\bar\Fm}
\xrightarrow{F_1-1}\bar P_{\bar\Fm}\longrightarrow 0.
\]
The proposition follows since the $p^n$-torsion of
$\hat G(\bar R)$ and $G(\bar K)$ coincide.
\end{proof}


\section{Some frame formalism}
\label{Se-frames}

Before we proceed we introduce a formal definition.
Let $\FFF=(S,R,I,\sigma,\sigma_1)$ be a frame in the 
sense of \cite{Lau-Frames} such that $S$ is a $\ZZ_p$-algebra
and $\sigma$ is $\ZZ_p$-linear.
For an $\FFF$-window $\PPP=(P,Q,F,F_1)$ we consider the 
\emph{module of invariants}
\[
T(\PPP)=\{x\in Q\mid F_1(x)=x\};
\]
this is a $\ZZ_p$-module. Let us list some of its formal properties.

Functoriality in $\FFF$:
Let $\alpha:\FFF\to\FFF'=(S',I',R',\sigma',\sigma_1')$ 
be a $u$-homomor\-phism of frames, thus $u\in S'$ is a unit,
and we have $\sigma_1'\alpha=u\cdot\alpha\sigma_1$ on $I$.
Assume that a unit $c\in S'$ with $c\sigma'(c)^{-1}=u$
is given. For an $\FFF$-window $\PPP$ as above,
the $S$-linear map $P\to S'\otimes_SP$,
$x\mapsto c\otimes x$ induces a $\ZZ_p$-linear map
\[
\tau(\PPP)=\tau_{c}(\PPP):T(\PPP)\to T(\alpha_*\PPP).
\]

Duality:
Recall that a bilinear form of $\FFF$-windows
$\gamma:\PPP\times\PPP'\to\PPP''$ is an $S$-bilinear map
$\gamma:P\times P'\to P''$ with $Q\times Q'\to Q''$ such
that for $x\in Q$ and $x'\in Q'$ we have 
$\gamma(F_1x,F_1'x')=F_1''(\gamma(x,x'))$. 
It induces a bilinear map of $\ZZ_p$-modules 
$T(\PPP)\times T(\PPP')\to T(\PPP'')$.
Let us denote the $\FFF$-window $(S,I,\sigma,\sigma_1)$
by $\FFF$ again. 
For each $\FFF$-window $\PPP$ there is a well-defined dual 
$\FFF$-window $\PPP^t$ together with a perfect bilinear form 
$\PPP\times\PPP^t\to\FFF$. It gives a bilinear map 
$T(\PPP)\times T(\PPP^t)\to T(\FFF)$. In our applications,
$T(\FFF)$ will be free of rank one, and the bilinear map
will turn out to be perfect.

Functoriality of duality:
For a $u$-homomorphism of frames $\alpha:\FFF\to\FFF'$ 
with $c$ as above and for a
bilinear form of $\FFF$-windows $\gamma:\PPP\times\PPP'\to\PPP''$,
the base change of $\gamma$ multiplied by $c^{-1}$ is a 
bilinear form of $\FFF'$-windows 
$\alpha_*\PPP\times\alpha_*\PPP'\to\alpha_*\PPP''$,
which we denote by $\alpha_*(\gamma)$;
see \cite[Lemma 2.14]{Lau-Frames}.
By passing to the modules of invariants
we obtain a commutative diagram
\[
\xymatrix@M+0.2em@C+1em{
T(\PPP)\times T(\PPP') \ar[r]^-{\gamma} 
\ar[d]_{\tau(\PPP)\times\tau(\PPP')} &
T(\PPP'') \ar[d]^{\tau(\PPP'')} \\
T(\alpha_*\PPP)\times T(\alpha_*\PPP') 
\ar[r]^-{\alpha_*(\gamma)} &
T(\alpha_*\PPP'').
}
\]
This will be applied to the bilinear form
$\PPP\times\PPP^t\to\FFF$.


\section{The case of perfect residue fields}
\label{Se-perfect}

Let $R,K,k,\Fm$ be as in Section \ref{Se-Connected}.
Assume that the residue field $k$ is perfect.
As in \cite[Sections 2.3 and 2.7]{Lau-Relation} we consider
the frame
\[
\DDD_R
=\varprojlim_n\DDD_{R/\Fm^n}
=(\WW(R),\II_R,R,f,\ff_1).
\]
Windows over $\DDD_R$, called Dieudonn\'e displays over $R$,
are equivalent to $p$-divisible groups $G$ over $R$ by 
\cite{Zink-DDisp} if $p$ is odd and by
\cite[Theorem A]{Lau-Relation} in general. 
The Tate module $T_p(G)$ can be expressed in terms of the 
associated Dieudonn\'e display by a variant of 
Proposition \ref{Pr-Tate-connected} as follows.

Let $R^{\nr}$ be the strict henselisation of $R$.
This is an excellent normal domain by
\cite{Greco-Excellent} or \cite{Seydi-Weierstrass},
so its completion $\hat R^{\nr}$ is a normal domain again.
Let $K^{\nr}\subset\hat K^{\nr}$ be the fraction fields
of $R^{\nr}\subset\hat R^{\nr}$,
let $\tilde K$ be an algebraic closure of $\hat K^{\nr}$,
and let $\tilde R\subset\tilde K$ be the integral closure
of $\hat R^{\nr}$. We define a frame
\[
\DDD_{\tilde R}=\varinjlim_E\varprojlim_n\DDD_{R_E/\Fm_E^n}
=(\WW(\tilde R),\II_{\tilde R},\tilde R,f,\ff_1)
\]
where $E$ runs through the finite extensions of $\hat K^{\nr}$
in $\tilde K$ and where $R_E\subset E$ is the integral closure
of $\hat R^{\nr}$.
Since $\tilde R$ has no $p$-torsion, the componentwise
$p$-adic completion of $\DDD_{\tilde R}$ is a frame again, 
which we denote by
\[
\hat\DDD_{\tilde R}=
(\hat\WW(\tilde R),\hat\II_{\tilde R},\hat{\tilde R},f,\ff_1).
\]

Let $\bar K\subset\tilde K$ be the algebraic closure of $K$
and let $\GGG_K=\Gal(\bar K/K)$.
The tensor product $\bar K\otimes_{K^{\nr}}\hat K^{\nr}$ is a 
subfield of $\tilde K$; see \cite[Chapitre IX, Th\'eor\`eme 3]{Raynaud-Henseliens}, 
with equality if $R$ is one-dimensional, since
in any case the etale coverings of the 
complements of the maximal ideals 
in $\Spec R^{\nr}$ and $\Spec\hat R^{\nr}$ coincide by 
\cite[Th\'eor\`eme 5]{Elkik} or by \cite[II 2.1]{Artin-Etale-Coverings}.
Let $\tilde\GGG_K$ be the group of automorphisms of $\tilde K$
whose restriction to $\bar K\hat K^{\nr}$ is induced by an
element of $\GGG_K$. This group acts naturally on 
$\DDD_{\tilde R}$ and on $\hat\DDD_{\tilde R}$.
The projection $\tilde\GGG_K\to\GGG_K$ is surjective, 
and bijective if $R$ is one-dimensional.

\begin{prop}
\label{Pr-Tate-perfect}
Let  $G$ be a $p$-divisible group over $R$ and let $\PPP=\Phi_R(G)$ 
be the Dieudonn\'e display over $R$ associated to $G$ in \cite{Lau-Relation}.
Let $\hat\PPP_{\!\tilde R}=(\hat P_{\tilde R},\hat Q_{\tilde R},F,F_1)$ 
be the base change of $\PPP$ to $\hat\DDD_{\tilde R}$.
There is a natural exact sequence of $\tilde\GGG_K$-modules
\[
0\longrightarrow T_p(G)\longrightarrow\hat Q_{\tilde R}\xrightarrow{F_1-1}
\hat P_{\tilde R}\longrightarrow 0.
\]
\end{prop}

\noindent
In particular we have an isomorphism of\/ $\GGG_K$-modules
\[
\per:T_p(G)\xrightarrow\sim T(\hat\PPP_{\!\tilde R})
\]
which we call the period isomorphism is display theory.

\begin{proof}[Proof of Proposition \ref{Pr-Tate-perfect}]
For a $p$-divisible group $G$ over $R$ and for finite 
extensions $E$ of $\hat K^{\nr}$ in $\tilde K$ we set
\[
\hat G(R_E)=\varprojlim_nG(R_E/\Fm_E^n);
\qquad
\hat G(\tilde R)=\varinjlim_E\hat G(R_E).
\]
Multiplication by $p$ is surjective on 
$\hat G(\tilde R)$ by Lemma \ref{Le-p-surj}
applied over $\hat R^{\nr}$.
The group $\tilde\GGG_K$ acts on the system $\hat G(R_E)$ 
for varying $E$ and thus on $\hat G(\tilde R)$.
The rings $R_{E,n}=R_E/\Fm_E^n$ are local Artin rings
with residue field $\bar k$. Thus $R_{E,n}$
lies in the category $\JJJ_{R/\Fm^n}$ 
defined in \cite[Section~9]{Lau-Relation}. 
Let $\PPP_{E,n}=(P_{E,n},Q_{E,n},F,F_1)$
be the base change of $\PPP$ to $R_{E,n}$.
Since every ind-\'etale covering of $\Spec R_{E,n}$
has a section, the definition of the functor
$\BT$ in \cite[Proposition~9.4]{Lau-Relation}
as an ind-\'etale cohomology sheaf together with \cite[Proposition 9.7]{Lau-Relation}
shows that $G(R_{E,n})=\BT(\PPP_{E,n})(R_{E,n})$ is quasi-isomorphic 
to the complex of $\tilde\GGG_K$-modules in degrees $-1,0,1$
\[
C_{E,n}=[Q_{E,n}\xrightarrow{F_1-1}P_{E,n}]
\otimes[\ZZ\to\ZZ[1/p]].
\]
Let
\[
C_E=\varprojlim_nC_{E,n}; \qquad
C=\varinjlim_EC_E;
\]
where $E$ runs through the finite extensions of $K^{\nr}$ in $\bar K$.
Since $G(R_{E,n})$ and the components of $C_{E,n}$ form
surjective systems with respect to $n$, the complex $C$
is quasi-isomorphic to $\hat G(\tilde R)$.
We will verify the following chain of isomorphisms 
(denoted $\cong$) and quasi-isomorphisms (denoted $\simeq$)
of complexes of $\tilde\GGG_K$-modules, which
will prove the proposition. 
Here $\Ext^1$ is taken component-wise in the second argument.
\begin{multline*}
T_p(G)\overset{(1)}\cong\Hom(\QQ_p/\ZZ_p,\hat G(\tilde R))
\overset{(2)}\simeq R\Hom(\QQ_p/\ZZ_p,\hat G(\tilde R)) \\
\overset{(3)}\simeq R\Hom(\QQ_p/\ZZ_p,C)
\overset{(4)}\simeq\Ext^1(\QQ_p/\ZZ_p,C[-1])
\overset{(5)}\cong[\hat Q_{\tilde R}
\xrightarrow{F_1-1}\hat P_{\tilde R}]
\end{multline*}
The torsion subgroups of $G(\bar K)$ and of $\hat G(\tilde R)$ coincide, 
which proves (1). For (2) we need that 
$\Ext^1(\QQ_p/\ZZ_p,\hat G(\tilde R))$ vanishes, which
is true since $p$ is surjective on $\hat G(\tilde R)$;
see Lemma \ref{Le-Ext1}.
The quasi-isomorphism between $\hat G(\tilde R)$ and $C$ gives (3).
Let $(P_{\tilde R},Q_{\tilde R},F,F_1)$ be the 
base change of $\PPP$ to $\DDD_{\tilde R}$ and let
$P_{\bar k}=W(\bar k)\otimes_{\WW(R)}P$.
The complex $C$ can be identified with the cone of
the map of complexes
\[
[Q_{\tilde R}\xrightarrow{F_1-1}P_{\tilde R}]
\longrightarrow
[P_{\bar k}[1/p]\xrightarrow{F_1-1}P_{\bar k}[1/p]].
\]
Since $\tilde R$ is a domain of characteristic zero, 
the ring $\WW(\tilde R)\subset W(\tilde R)$ has
no $p$-torsion, and thus the components 
of $C$ have no $p$-torsion. In particular,
$\Hom(\QQ_p/\ZZ_p,C)$ vanishes, which proves (4).
The $p$-adic completions of $P_{\tilde R}$ and $Q_{\tilde R}$
are $\hat P_{\tilde R}$ and $\hat Q_{\tilde R}$, while the
$p$-adic completion of $P_{\bar k}[1/p]$ is zero.
Thus Lemma \ref{Le-Ext1} gives (5).
\end{proof}


\section{A variant for the prime $2$}
\label{Se-variant-2}

We keep the notation and assumptions
of Section \ref{Se-perfect} and assume that $p=2$. 
One may ask what the preceding constructions
give if $\WW$ and $\DDD$ are replaced by their $v$-stabilised
variants $\WW^+$ and $\DDD^+$.
Recall that $\WW^+(R)=\WW(R)[v(1)]$ as a subring of $W(R)$, 
and we have a frame $\DDD^+_R=(\WW^+(R),\II^+_R,R,f,f_1)$ 
where $f_1$ is the inverse of $v$.
The $\WW(R)$-module $\WW^+(R)/\WW(R)$ 
is a one-dimensional $k$-vector space generated by $v(1)$; 
see \cite[Sections 1.4 and 2.5]{Lau-Relation}.
We put
\[
\DDD^+_{\tilde R}=\varinjlim_E\varprojlim_n\DDD^+_{R_E/\Fm^n_E}
=(\WW^+(\tilde R),\II^+_{\tilde R},\tilde R,f,f_1)
\]
with $E$ as in Section \ref{Se-perfect}, and denote 
the $2$-adic completion of $\DDD^+_{\tilde R}$ by
\[
\hat\DDD^+_{\tilde R}=(\hat\WW^+(\tilde R),\hat\II^+_{\tilde R},
\hat{\tilde R},f,f_1).
\]
For a $2$-divisible group $G$ over $R$ let $G^m$ be the
multiplicative part of $G$ and define $G^+$
by the following homomorphism of exact sequences of $2$-divisible groups:
\[
\xymatrix@M+0.2em{
0 \ar[r] &
G^m \ar[r] \ar[d]^p &
G \ar[r] \ar[d] &
G^u \ar[r] \ar@{=}[d] & 0 \\
0 \ar[r] &
G^m \ar[r] &
G^+ \ar[r] &
G^u \ar[r] & 0
}
\]

\begin{prop}
\label{Pr-Tate-perfect+}
Let  $G$ be a $2$-divisible group over $R$ with associated
Dieu\-donn\'e display $\PPP=\Phi_R(G)$. 
Let $\hat\PPP^+_{\!\tilde R}=
(\hat P^+_{\tilde R},\hat Q^+_{\tilde R},F,F_1^+)$ 
be the base change of $\PPP$ to $\hat\DDD^+_{\tilde R}$. 
There is a natural exact sequence of $\tilde\GGG_K$-modules
\[
0\longrightarrow T_2(G^+)\longrightarrow\hat Q^+_{\tilde R}\xrightarrow{F_1^+-1}
\hat P^+_{\tilde R}\longrightarrow 0.
\]
\end{prop}

\noindent
In particular we have an isomorphism of\/ $\GGG_K$-modules
\[
\per^+:T_2(G^+)\xrightarrow\sim T(\hat\PPP^+_{\!\tilde R}).
\]

\begin{proof}
Let $\bar P_{\bar k}=\bar k\otimes_{\WW(R)}P$. 
We will construct the following commutative diagram
with exact rows, where $\bar F$ is induced by $F$:
\[
\xymatrix@M+0.2em{
0 \ar[r] &
\hat Q_{\tilde R} \ar[r] \ar[d]^{F_1-1} &
\hat Q^+_{\tilde R} \ar[r] \ar[d]^{F_1^+-1} &
\bar P_{\bar k} \ar[r] \ar[d]^{\bar F-1} & 0 \\
0 \ar[r] &
\hat P_{\tilde R} \ar[r] &
\hat P^+_{\tilde R} \ar[r] &
\bar P_{\bar k} \ar[r]  & 0 
}
\]
Here the Frobenius-linear endomorphism $\bar F$ is nilpotent 
if $G$ is unipotent, and is given by an invertible matrix if 
$G$ is of multiplicative type. Thus $\bar F-1$ is surjective
with kernel an $\FF_p$-vector space of dimension equal
to the height of $G^m$, and Proposition \ref{Pr-Tate-perfect+}
follows from Proposition \ref{Pr-Tate-perfect} if the diagram is constructed.

The natural homomorphism $\hat\WW(\tilde R)\to\hat\WW^+(\tilde R)$
is injective and defines a $u_0$-homomorphism of frames
$\iota:\hat\DDD_{\tilde R}\to\hat\DDD^+_{\tilde R}$ where the unit
$u_0\in\WW^+(\ZZ_2)$ is defined by $v(u_0)=2-[2]$; 
see \cite[Section 2.5]{Lau-Relation}.
Since $u_0$ maps to $1$ in $W(\FF_2)$
there is a unique unit $c_0$ of $\WW^+(\ZZ_2)$ which
maps to $1$ in $W(\FF_2)$ such that 
$c_0f(c_0^{-1})=u_0$, namely $c_0=u_0f(u_0)f^2(u_0)\cdots$; 
see the proof of \cite[Proposition 8.7]{Lau-Frames}.

The cokernel of $\iota$ is given by
\begin{equation}
\label{Eq-quot}
\hat\II^+_{\tilde R}/\hat\II_{\tilde R}
=\hat\WW^+(\tilde R)/\hat\WW(\tilde R)
=\bar k\cdot v(1);
\end{equation}
see \cite[Lemma 1.10]{Lau-Relation}.
Extend the operator $\ff_1$ of $\hat\DDD_{\tilde R}$ 
to $\hat\DDD^+_{\tilde R}$ by $\ff_1=u_0^{-1}f_1$. Then
$\ff_1$ induces an $f$-linear endomorphism $\bar\ff_1$
of $\bar k\cdot v(1)$. 
We claim that $\bar\ff_1(v(1))=v(1)$.
It suffices to prove this formula in
$\WW^+(\ZZ_2)/\WW(\ZZ_2)\cong\FF_2$,
and thus it suffices to show that $\ff_1(v(1))$
does not lie in $\WW(\ZZ_2)$. But $\WW(\ZZ_2)$ is stable
under the operator $x\mapsto\vv(x)=v(u_0x)$, 
and $\vv(\ff_1(v(1))=v(1)$ does not lie in $\WW(\ZZ_2)$. 
This proves the claim.

Let us extend the operator $F_1$ of $\hat\PPP_{\!\tilde R}$ 
to $\hat\PPP^+_{\!\tilde R}$ by $F_1=u_0^{-1}F_1^+$. 
Then we have $c_0(F_1-1)=(F_1^+-1)c_0$ as
homomorphisms $\hat Q^+_{\tilde R}\to\hat P^+_{\tilde R}$,
and it suffices to construct the above diagram 
with $F_1$ in place of $F_1^+$.
Now \eqref{Eq-quot} implies that
$\hat Q^+_{\tilde R}/\hat Q_{\tilde R}=
\hat P^+_{\tilde R}/\hat P_{\tilde R}=\bar P_{\bar k}\cdot v(1)$,
which gives the exact rows. Clearly the left hand square commutes.
The relation $F_1(ax)=\ff_1(a)F(x)$ for 
$x\in\hat P^+_{\tilde R}$ and $a\in\hat\II^+_{\tilde R}$
applied with $a=v(1)$, together with $\bar\ff_1(v(1))=v(1)$ 
shows that the right hand square commutes.
\end{proof}

\begin{rmk}
\label{Re-per-per+}
The period isomorphisms $\per$ and $\per^+$ satisfy
$\per^+=\tau_{c_0}\circ\per$, where 
$\tau_{c_0}:T(\hat\PPP_{\!\tilde R})\to T(\hat\PPP^+_{\!\tilde R})$ 
is the homomorphism defined in Section \ref{Se-frames}.
\end{rmk}


\section{The relation with $A_{\cris}$}
\label{Se-Acris}

Let $R$ be a complete discrete valuation ring
with perfect residue field $k$ of characteristic $p$ and
fraction field $K$ of characteristic zero.
In this case our ring $\hat{\tilde R}$ is equal to
$\hat{\bar R}$, the $p$-adic completion of the integral
closure of $R$ in $\bar K$. 
Let $A_{\cris}=A_{\cris}(R)$ and consider the frame%
\footnote{Here we need
that $\sigma_1(\Fil A_{\cris})$ generates $A_{\cris}$.
But $\xi=p-[\underline p]$ lies in $\Fil A_{\cris}$, 
and $\sigma_1(\xi)=1-[\underline p]^p/p$ is a unit
because $[\underline p]$ lies in the divided power
ideal $\Fil A_{\cris}+pA_{\cris}$.}
\[
\AAA_{\cris}=(A_{\cris},\Fil A_{\cris},\hat{\bar R},\sigma,\sigma_1)
\]
with $\sigma_1=p^{-1}\sigma$.
For a $p$-divisible group $G$ over $R$ let $\DD(G)$ be its
covariant Dieudonn\'e crystal. The free $A_{\cris}$-module
$M=\DD(G_{\!\hat{\bar R}})_{A_{\cris}}$ carries a
filtration $\Fil M$ and a $\sigma$-linear endomorphism $F$.
The operator $F_1=p^{-1}F$ is well-defined on $\Fil M$, 
and we get an $\AAA_{\cris}$-window 
$\MMM=(M,\Fil M,F,F_1)$; see \cite[A.2]{Kisin-crys} or 
\cite[Proposition 3.15]{Lau-Relation}.
Faltings \cite{Faltings-Integral} constructs a period homomorphism
\[
\per_{\cris}:T_p(G)\to\Fil M^{F=p}=T(\MMM)
\]
which is bijective if $p$ is odd; for $p=2$ the homomorphism
is injective with cokernel annihilated by $p$. More
precisely, for $p=2$ the cokernel is zero if $G$ is 
unipotent by \cite[Proposition 1.1.10]{Kisin-2adic}, 
while the cokernel is non-zero if $G$ is non-zero
and of multiplicative type; thus the period homomorphism 
extends to an isomorphism $T_p(G^+)\cong T(\MMM)$ with
$G^+$ as in Section \ref{Se-variant-2}.

Let us relate this with the period isomorphisms 
of Sections \ref{Se-perfect} and \ref{Se-variant-2}.
For the sake of uniformity, in the following we write
$\WW^+=\WW$ etc.\ if $p$ is odd.
Then $\hat\WW^+({\tilde R})\to\hat{\bar R}$ is a divided
power extension of $p$-adic rings for all $p$.
Recall that $A_{\cris}$ is the $p$-adic completion
of the divided power envelope of the kernel of 
$\theta:A_{\inf}\to\hat{\bar R}$, where $A_{\inf}=W(\RRR)$,
and $\RRR$ is the projective limit of $\bar R/p\bar R$
under Frobenius.

\begin{lemma}
\label{Le-kappa-inf}
There are unique homomorphisms $\varkappa_{\inf}$
and $\varkappa_{\cris}$ of extensions of $\hat{\bar R}$ 
as written below. They commute with Frobenius, and the
diagram commutes.
\[
\xymatrix@M+0.2em@C+1em{
A_{\inf} \ar[d] \ar[r]^-{\varkappa_{\inf}} &
\hat\WW(\tilde R) \ar[d] \\
A_{\cris} \ar[r]^-{\varkappa_{\cris}} &
\hat\WW^+(\tilde R)
}
\]
\end{lemma}

\begin{proof}
Briefly said, the universal property of $A_{\cris}$ gives
$\varkappa_{\cris}$, and the lemma explicates its construction.
Namely, all elements $x$ of the
kernel of $\hat\WW^+(\tilde R)/p^n\to\hat{\bar R}/p$
satisfy $x^{p^n}=0$ due to the divided powers on this ideal.
Since the cokernel of the inclusion $\hat\WW(\tilde R)\to
\hat\WW^+(\tilde R)$ is annihilated by $p$, 
for all $x$ in the kernel of 
$\hat\WW(\tilde R)/p^n\to\hat{\bar R}/p$
we get $x^{p^{n+1}}=0$.
Thus the universality of the Witt vectors gives a unique
homomorphism $\varkappa_{\inf}$ of extensions of $\hat{\bar R}/p$,
and the universality also implies that $\varkappa_{\inf}$ 
commutes with the Frobenius and with the projections 
to $\hat{\bar R}$.
Since $\hat\WW^+({\tilde R})\to\hat{\bar R}$ is a divided
power extension of $p$-adic rings, $\varkappa_{\inf}$ 
extends uniquely to $\varkappa_{\cris}$, and $\varkappa_{\cris}$
commutes with the Frobenius because this holds for $\varkappa_{\inf}$.
\end{proof}

Using that $\hat\WW^+(\tilde R)$ has no $p$-torsion, 
it follows that $\varkappa_{\cris}$ 
is a $\tilde\GGG_K$-equivariant strict frame homomorphism
\[
\varkappa_{\cris}:\AAA_{\cris}\to\hat\DDD^+_{\tilde R}.
\]

Now let $\PPP=\Phi_R(G)$ be the Dieudonn\'e display associated to $G$. 
The crystal $\DD(G)$ gives rise to a $\DDD^+_{R}$-window $\Phi^+_R(G)$
by \cite[Theorem 3.17]{Lau-Relation} if $p$ is odd and by
\cite[Proposition 3.21]{Lau-Relation} if $p=2$.
Its base change to $\hat\DDD^+_{\tilde R}$ is isomorphic
to $\varkappa_{\cris*}(\MMM)$ by the functoriality of $\DD(G)$.
Let $\iota:\DDD_R\to\DDD^+_R$ be the inclusion, which is the
identity when $p$ is odd.
We have an isomorphism of $\DDD_R^+$-windows
$\iota_*(\PPP)\cong\Phi_R^+(G)$ by definition if $p$ is odd
and by \cite[Theorem 4.9]{Lau-Relation} if $p=2$.
Thus we get an isomorphism of $\hat\DDD_{\tilde R}^+$-windows
$\hat\PPP^+_{\tilde R}\cong\varkappa_{\cris*}(\MMM)$, which induces
a homomorphism of $\GGG_K$-modules
\[
\tau:T(\MMM)\to T(\hat\PPP^+_{\!\tilde R})
\]
as explained in Section \ref{Se-frames}.
 
\begin{prop}
\label{Pr-commute}
The following diagram of $\GGG_K$-modules commutes
up to multiplication by a $p$-adic unit which is
independent of $G$, and $\tau$ is an isomorphism.
\[
\xymatrix@C+1em@M+0.2em{
T_p(G) \ar[r]^{\per_{\cris}} \ar[d]_{\per} & T(\MMM) \ar[d]^\tau \\
T(\hat\PPP_{\!\tilde R}) \ar[r]^{\tau_{c_0}} & 
T(\hat\PPP^+_{\!\tilde R}).
}
\]
\end{prop}

\begin{proof}[Proof of Proposition \ref{Pr-commute}]
The composition $\tau_{c_0}\circ\per$ extends to an isomorphism
$T_p(G^+)\cong T(\hat\PPP^+_{\tilde R})$ by 
Proposition \ref{Pr-Tate-perfect+} and Remark \ref{Re-per-per+}.
Thus if the diagram commutes, by the properties of $\per_{\cris}$
recalled above it follows that $\tau$ is an isomorphism.
Here we need only that the $\ZZ_p$-rank of $T(\MMM)$ is
$\le$ the height of $G$ and that  $\per_{\cris}$ is not bijective
when $p=2$ and $G$ is non-zero of multiplicative type.

Let us prove that the diagram commutes.
We start with the case $G=\QQ_p/\ZZ_p$.
Then $\per$ and $\tau_{c_0}$ are isomorphisms by
Propositions \ref{Pr-Tate-perfect} and \ref{Pr-Tate-perfect+}.
We have $T_p(G)=\ZZ_p$ and $M=\Fil M=A_{\cris}$ 
with Frobenius $p\sigma$ so that 
$T(\MMM)=A_{\cris}^{\sigma=1}=\ZZ_p$,
and $\per_{\cris}$ is the identity of $\ZZ_p$ by definition. 
Moreover $\hat Q^+_{\tilde R}=\hat P^+_{\tilde R}=\hat\WW^+_{\tilde R}$ 
with $F_1=f$, so $\tau$ can be identified with the homomorphism
$A_{\cris}^{\sigma=1}\to\hat\WW^+(\tilde R)^{f=1}$
induced by $\varkappa_{\cris}$.
This is a homomorphism of $\ZZ_p$-algebras with
source $\ZZ_p$ and target free of rank $1$ as a $\ZZ_p$-module
by Proposition \ref{Pr-Tate-perfect+}. Thus $\tau$ is bijective,
and it follows that $\tau_{c_0}\circ\per=\rho\cdot\tau\circ\per_{\cris}$
for a well defined $\rho\in\ZZ_p^*$.

Let now $G$ be arbitrary. Since the map $\tau_{c_0}\circ\per=\per^+$ 
is injective with cokernel annihilated by $p$, the composition
$\gamma=p\rho\cdot(\per^+)^{-1}\circ\tau\circ\per_{\cris}$ is a 
well-defined functorial endomorphism of $T_pG$. 
We have to show that $\gamma=p$. By \cite[4.2]{Tate}, 
$\gamma$ comes from an endomorphism $\gamma_G$ of $G$;
moreover $\gamma_G$ is functorial in $G$ and compatible with 
finite extensions of the base ring $R$ inside $\bar K$. 
The endomorphisms $\gamma_G$ induce a functorial endomorphism
$\gamma_H$ of each commutative finite flat $p$-group scheme $H$ 
over a finite extension $R'$ of $R$ inside $\bar K$ 
because $H$ can be embedded into a $p$-divisible group by
Raynaud \cite[3.1.1]{BBM}; cf.\ \cite[2.3.5]{Kisin-crys} 
or \cite[Proposition 4.1]{Lau-Relation}.
Assume that $H$ is annihilated by $p^r$ and let $H_0=\ZZ/p^r\ZZ$.
There is a finite extension $R''$ of $R'$ inside $\bar K$ 
such that $H(\bar K)=H(R'')=\Hom_{R''}(H_0,H)$.
Since $\gamma_{H_0}=p$ it follows that $\gamma_H=p$,
and thus $\gamma_G=p$ for all $G$.
\end{proof}


\section{The ring $\FS^{\nr}$}
\label{Se-FS-nr}

Let us recall the ring $\FS^{\nr}$ of \cite{Kisin-crys},
which is denoted $A^+_S$ in \cite{Fontaine-90}.
One starts with a two-dimensional complete regular local ring 
$\FS$ of characteristic zero
with perfect residue field $k$ of characteristic $p$ 
equipped with a Frobenius lift $\sigma:\FS\to\FS$.
Let $\delta:\FS\to W(\FS)$ be the unique ring homomorphism 
with $\delta\sigma=f\delta$ and $w_0\delta=\id$, and
let $t$ be a generator of the kernel of the composition
$\FS\to W(\FS)\to W(k)$.
Then $\FS=W(k)[[t]]$ and $\sigma(t)\in t\FS$.

Let $\OOO_{\EEE}$ be the $p$-adic completion of
$\FS[t^{-1}]$ and let $\EE=k((t))$ be its residue field.
Fix a maximal unramified extension $\OOO_{\EEE^{\nr}}$
of $\OOO_\EEE$ and let $\OOO_{\widehat{\EEE^{\nr}}}$
be its $p$-adic completion.
Let $\EE^{\sep}$ be the residue field of $\OOO_{\EEE^{\nr}}$, 
let $\bar\EE$ be an algebraic closure of $\EE^{\sep}$, 
let $\OOO_\EE=\FS/p\FS=k[[t]]$, and let 
$\OOO_{\bar\EE}\subset\bar\EE$ be its integral closure.
The Frobenius lift $\sigma$ on $\FS$ extends
uniquely to $\OOO_{\widehat{\EEE^{\nr}}}$ 
and induces an embedding 
\[
\OOO_{\widehat{\EEE^{\nr}}}\xrightarrow\delta 
W(\OOO_{\widehat{\EEE^{\nr}}})\to W(\bar\EE)
\]
with $\delta$ as above. One defines
$\FS^{\nr}=\OOO_{\widehat{\EEE^{\nr}}}\cap W(\OOO_{\bar\EE})$
inside $W(\bar\EE)$.
This ring is stable under $\sigma$,
and $\FS^{\nr}=\varprojlim\FS^{\nr}_n$ with
$\FS^{\nr}_n=(\OOO_{\EEE^{\nr}}/p^n\OOO_{\EEE^{\nr}})
\cap W_n(\OOO_{\bar\EE})$ inside $\WW_n(\bar\EE)$.
By \cite[B 1.8.3]{Fontaine-90} we have 
$\FS^{\nr}_n=\FS^{\nr}/p^n\FS^{\nr}$, in particular
$\FS^{\nr}$ is $p$-adic. 

\medskip
\begin{small}
For completeness we sketch a proof of the relation
$\FS^{\nr}_n=\FS^{\nr}/p^n\FS^{\nr}$.
It is easy to see that there is a commutative diagram 
with exact rows and injective vertical maps:
\[
\xymatrix@M+0.2em{
0 \ar[r] & \FS^{\nr}/p^{n-1}\FS^{\nr} \ar[r]^-p \ar[d] &
\FS^{\nr}/p^n\FS^{\nr} \ar[r] \ar[d] & 
\FS^{\nr}/p\FS^{\nr} \ar[r] \ar[d] & 0 \\
0 \ar[r] & \FS^{\nr}_{n-1} \ar[r]^-p &
\FS^{\nr}_n \ar[r] & \FS^{\nr}_1 
}
\]
By induction it suffices to show that the projection
$\FS^{\nr}\to\FS_1^{\nr}$ is surjective.
Now $\FS_1^{\nr}=\OOO_{\EE^{\sep}}$ is the union of
$\OOO_\FF$ over all finite separable extensions $\FF/\EE$,
and in each case $(\OOO_\FF,\sigma)$ is a free
Frobenius module over
$\OOO_\EE$ that becomes \'etale over $\EE$. Let 
$M=(\FS^r,\phi)$ be a Frobenius module over
$\FS$ that lifts $(\OOO_\FF,\sigma)$, so $M$ becomes \'etale
over $\OOO_{\EEE}$. The projection
\[
\Hom_{\FS,\phi}(M,\OOO_{\widehat{\EEE^{\nr}}})\to
\Hom_{\FS,\phi}(M,\EE^{\sep})=
\Hom_{\OOO_\EE,\sigma}(\OOO_\FF,\EE^{\sep})
\]
is surjective by \cite[A 1.2]{Fontaine-90}. 
Let $\alpha:\FS^r\to\OOO_{\widehat{\EEE^{\nr}}}$
be a lift of the inclusion $\OOO_\FF\to\OOO_{\EE^{\sep}}$.
We claim that
$\Hom_{\FS,\phi}(M,\OOO_{\widehat{\EEE^{\nr}}})=
\Hom_{\FS,\phi}(M,W(\bar\EE))=
\Hom_{\FS,\phi}(M,W(\OOO_{\bar\EE}))$;
then this group also coincides with
$\Hom_{\FS,\phi}(M,\FS^{\nr})$,
so $\alpha$ factors over $\FS^{\nr}$,
which implies that $\FS^{\nr}\to\FS_1^{\nr}$
is surjective as desired.
The first equality is clear by \cite[A 1.2]{Fontaine-90}.
For the second equality it suffices that
$\Hom_{\FS,\phi}(M,W(\bar\EE)/W(\OOO_{\bar\EE}))$
vanishes; this is true because every finitely generated
Frobenius-stable $\FS$-module in $W(\bar\EE)/W(\OOO_{\bar\EE})$
is zero.
\end{small}


\section{Breuil-Kisin modules}

Let $R$ be a complete discrete valuation ring with
perfect residue field $k$ of characteristic $p$ and fraction
field $K$ of characteristic zero. 
Let $\FS=W(k)[[t]]$ and let $\sigma:\FS\to\FS$ be a 
Frobenius lift that stabilises the ideal $t\FS$.
We choose a representation $R=\FS/E\FS$ where 
$E$ has constant term $p$. 
Let $\pi\in R$ be the image of $t$, so
$\pi$ generates the maximal ideal of $R$.

For an $\FS$-module $M$ let
$M^{(\sigma)}=\FS\otimes_{\sigma,\FS}M$.
We consider pairs $(M,\phi)$ where $M$ is a finite $\FS$-module
and where $\phi:M\to M^{(\sigma)}$ is an $\FS$-linear map 
with cokernel annihilated by $E$.
Following the \cite{Vasiu-Zink} terminology, $(M,\phi)$
is called a Breuil window (resp.\ a Breuil module) relative
to $\FS\to R$ if the $\FS$-module $M$ is free 
(resp.\ annihilated by a power of $p$ and of 
projective dimension at most one).
We have a frame in the sense of \cite{Lau-Frames}
\[
\BBB=(\FS,E\FS,R,\sigma,\sigma_1)
\]
with $\sigma_1(Ex)=\sigma(x)$ for $x\in\FS$.
Windows $\PPP=(P,Q,F,F_1)$ over $\BBB$ are equivalent to 
Breuil windows relative 
to $\FS\to R$ by the functor $\PPP\mapsto(Q,\phi)$
where $\phi:Q\to Q^{(\sigma)}$ 
is the composition of the inclusion $Q\to P$ with the
inverse of the isomorphism $Q^{(\sigma)}\cong P$ defined by 
$a\otimes x\mapsto aF_1(x)$. 

As in \cite[Section 6]{Lau-Relation} let $\varkappa$ be the ring homomorphism
\[
\varkappa:\FS\xrightarrow\delta W(\FS)\to W(R).
\]
Its image lies in $\WW(R)$ if and only if the endomorphism
of $t\FS/t^2\FS$ induced by $\sigma$ is divisible by $p^2$.
In this case, $\varkappa:\FS\to\WW(R)$ is a 
$\uu$-homomorphism of frames $\BBB\to\DDD_R$ for the
unit $\uu=\ff_1(\varkappa(E))$ of $\WW(R)$,
and $\varkappa$ induces an equivalence between
$\BBB$-windows and $\DDD_R$-windows, which are
equivalent to $p$-divisible groups over $R$.
As a consequence, Breuil modules relative to $\FS\to R$ are 
equivalent to commutative finite flat $p$-group schemes over $R$. 

Since $\uu$ maps to $1$ 
under $\WW(R)\to W(k)$, there is a unique
unit $\cc\in\WW(R)$ which maps to $1$ in $W(k)$
with $\cc\sigma(\cc^{-1})=\uu$. It is given by
$\cc=\uu\sigma(\uu)\sigma^2(\uu)\cdots$;
see the proof of \cite[Proposition 8.7]{Lau-Frames}.

\subsection{Modules of invariants}

For a Breuil module or Breuil window $(M,\phi)$
relative to $\FS\to R$ we write
$M^{\nr}=\FS^{\nr}\otimes_\FS M$ and 
$M^{\nr}_\EEE=\OOO_{\widehat{\EEE^{\nr}}}\otimes_{\FS}M$.
Consider the $\ZZ_p$-modules:
\begin{gather*}
T^{\nr}(M,\phi)=\{x\in M^{\nr}\mid\phi(x)=1\otimes x
\text{ in }\FS^{\nr}\otimes_{\sigma,\FS^{\nr}}M^{\nr}\} \\
T^{\nr}_\EEE(M,\phi)=\{x\in M^{\nr}_\EEE\mid\phi(x)=1\otimes x
\text{ in }\OOO_{\widehat{\EEE^{\nr}}}\otimes_{\sigma,\OOO_{\widehat{\EEE^{\nr}}}}M^{\nr}_\EEE\}
\end{gather*}
By \cite[A 1.2]{Fontaine-90}, 
$T^{\nr}_\EEE(M,\phi)$ is finitely generated, and the natural map
\begin{equation}
\label{Eq-A2}
\OOO_{\widehat{\EEE^{\nr}}}\otimes_{\ZZ_p}T^{\nr}_\EEE(M,\phi)
\to\OOO_{\widehat{\EEE^{\nr}}}\otimes_\FS M
\end{equation}
is bijective. It is pointed out in \cite{Kisin-crys,Kisin-2adic}
that the natural map
\begin{equation}
\label{Eq-B1}
T^{\nr}(M,\phi)\to T^{\nr}_\EEE(M,\phi)
\end{equation}
is bijective as well. If $(M,\phi)$ is a Breuil window, 
this follows from the proof of \cite[B 1.8.4]{Fontaine-90}. 
If $(M,\phi)$ is a Breuil module, the  map \eqref{Eq-B1} is
injective since the group
$X=\OOO_{\widehat{\EEE^{\nr}}}/\FS^{\nr}$ has
no $p$-torsion and thus $\Tor_1^\FS(X,M)$ is zero.
One can find a Breuil window $(M',\phi')$ and a surjective map 
$(M',\phi')\to(M,\phi)$. Then $T^{\nr}(M',\phi')\cong 
T^{\nr}_\EEE(M',\phi')\to T^{\nr}_\EEE(M,\phi)$ is surjective,
thus \eqref{Eq-B1} is surjective.

\subsection{The choice of $K_\infty$.}

Let $\hat{\bar\Fm}$ be the maximal ideal of $\hat{\bar R}$.
The power series $\sigma(t)$ defines a map 
$\sigma(t):\hat{\bar\Fm}\to\hat{\bar\Fm}$.
This map is surjective, and the inverse images of
algebraic elements are algebraic by
the Weierstrass preparation theorem.
Choose a system of elements $(\pi^{(n)})_{n\ge 0}$ of $\bar K$
with $\pi^{(0)}=\pi$ and $\sigma(t)(\pi^{(n+1)})=\pi^{(n)}$, and 
let $K_\infty$ be the extension of $K$ generated by all $\pi^{(n)}$. 
The system $(\pi^{(n)})$ corresponds to an element 
$\underline\pi\in\RRR=\varprojlim\bar R/p\bar R$,
the limit taken with respect to Frobenius. 

We embed $\OOO_\EE=k[[t]]$ into $\RRR$ by $t\mapsto\underline\pi$,
and identify $\EE^{\sep}$ and $\bar\EE$ with subfields of 
$\Frac\RRR$; thus $W(\bar\EE)\subset W(\Frac\RRR)$.
Then $\FS^{\nr}=\OOO_{\widehat{\EEE^{\nr}}}\cap W(\RRR)$,
and the unique ring homomorphism $\theta:W(\RRR)\to\hat{\bar R}$ 
which lifts the projection $W(\RRR)\to\bar R/p\bar R$
induces a homomorphism 
\[
pr^{\nr}:\FS^{\nr}\to\hat{\bar R}.
\]
Let us verify that its restriction to $\FS$ is the
given projection $\FS\to R$.

\begin{lemma}
\label{Le-pr-nr}
We have $pr^{\nr}(t)=\pi$.
\end{lemma}

\begin{proof}
The lemma is easy if $\sigma(t)=t^p$ since then $\delta(t)=[t]$ 
in $W(\FS)$, which maps to $[\underline\pi]$ in $W(\RRR)$, 
and $\theta([\underline\pi])=\pi$ in this case. 
In general let $\delta(t)=(g_0,g_1,\dots)$ with $g_i\in\FS$; 
these power series are determined by the relations 
\[
g_0^{p^n}+pg_1^{p^{n-1}}+\cdots+p^ng_n=\sigma^n(t)
\]
for $n\ge 0$. Let $x=(x_0,x_1,\ldots)\in W(\RRR)$ 
be the image of $t$, thus $x_i=g_i(\underline\pi)$.
Write $x_i=(x_{i,0},x_{i,1},\ldots)$ with 
$x_{i,n}=g_i(\underline\pi)_n\in\bar R/p\bar R$.
If $\tilde x_{i,n}\in\hat{\bar R}$ lifts $x_{i,n}$ we have
\[
pr^{\nr}(t)=\theta(x)=\lim_{n\to\infty}
\bigl((\tilde x_{0,n})^{p^n}+
p(\tilde x_{1,n})^{p^{n-1}}+\cdots+
p^n\tilde x_{n,n}\bigr).
\]
If we choose $\tilde x_{i,n}=g_i(\pi^{(n)})$, the sum in the limit 
becomes $\sigma^n(t)(\pi^{(n)})=\pi$, and the lemma is proved.
\end{proof}

Since the natural action of $\GGG_{K_\infty}=\Gal(\bar K/K_\infty)$
on $W(\Frac \RRR)$ is trivial on $\OOO_\EEE$ it
stabilises $\OOO_{\hat\EEE^{\nr}}$ and $\FS^{\nr}$
with trivial action on $\FS$.
Thus $\GGG_{K_\infty}$ acts on $T^{\nr}(M,\phi)$
for each Breuil window or Breuil module $(M,\phi)$.

\subsection{From $\FS^{\nr}$ to Zink rings}

The composition of the inclusion
$\FS^{\nr}\to W(\RRR)$ chosen above with the homomorphism
$\varkappa_{\inf}:W(\RRR)\to\hat\WW(\tilde R)$ from 
Lemma \ref{Le-kappa-inf} is a ring homomorphism
\[
\varkappa^{\nr}:\FS^{\nr}\to\hat\WW(\tilde R)
\]
that commutes with Frobenius and with the projections
to $\hat{\bar R}$.

\begin{lemma}
\label{Le-kappa-nr}
If the image of $\varkappa:\FS\to W(R)$ lies in
$\WW(R)$, i.e.\ if the endomorphism of $t\FS/t^2\FS$
induced by $\sigma$ is divisible by $p^2$,
the following diagram of rings commutes;
the vertical maps are the obvious inclusions.
\[
\xymatrix@M+0.2em@C+1em{
\FS \ar[r]^-\varkappa \ar[d] & \WW(R) \ar[d]^\iota \\ 
\FS^{\nr} \ar[r]^-{\varkappa^{\nr}} & \hat\WW(\tilde R)
}
\] 
\end{lemma}

\begin{proof}
The assumption $\varkappa(\FS)\subset\WW(R)$
is equivalent to $\delta(\FS)\subset\WW(\FS)$;
see \cite[Proposition 6.2]{Lau-Relation}.
As in the proof of Lemma \ref{Le-pr-nr} we write
$\delta(t)=(g_0,g_1,\ldots)$ with $g_i\in\FS$.
We have to show that 
\[
\varkappa_{\inf}((g_0(\underline\pi),g_1(\underline\pi),\ldots))
=
\iota((g_0(\pi),g_1(\pi),\ldots))
\]
in $\hat\WW(\tilde R)$.
Again, if $y_{i,n}\in\hat\WW(\tilde R)$ is a lift
of $x_{i,n}=g_i(\underline\pi)_n\in\bar R/p\bar R$, 
the left hand side of this equation is equal to
\[
\lim_{n\to\infty}
\bigl((y_{0,n})^{p^n}+p(y_{1,n})^{p^{n-1}}
+\cdots+p^ny_{n,n}\bigr).
\]
We will choose $y_{i,n}\in\WW(\tilde R)$ (no $p$-adic completion)
such that the sum in the limit is equal to 
$(g_0(\pi),g_1(\pi),\ldots)$ in $\WW(\tilde R)$;
this will prove the lemma.
In the special case $\sigma(t)=t^p$, thus $g_0=t$ and $g_i=0$
for $i\ge 1$, we can simply take $y_{0,n}=[\pi^{(n)}]$
and $y_{i,n}=0$ for $i\ge 1$; then the calculation is trivial.
In general, let $\delta(g_i)=(h_{i,0},h_{i,1},\ldots)$ in $\WW(\FS)$,
so the power series $h_{i,j}$ are determined by the equations
\[
h_{i,0}^{p^m}+ph_{i,1}^{p^{m-1}}+\cdots+p^mh_{i,m}=\sigma^m(g_i)
=g_i(\sigma^m(t))
\]
for $m\ge 0$,
and put $y_{i,n}=(h_{i,0}(\pi^{(n)}),h_{i,1}(\pi^{(n)}),\ldots)
\in\WW(\tilde R)$. Since the Witt polynomials 
$w_m(X_0,\ldots,X_m)=X_0^{p^m}+\cdots+p^mX_m$ 
for $m\ge 0$ define an injective map
$\WW(\tilde R)\subset W(\tilde R)\to\tilde R^\infty$, 
we have to show that for $n,m\ge 0$ the following holds.
\[
w_m((y_{0,n})^{p^n}+p(y_{1,n})^{p^{n-1}}+\cdots+p^ny_{n,n})
=w_m((g_0(\pi),g_1(\pi),\ldots))
\]
The right hand side is equal to $\sigma^m(t)(\pi)$.
Since $w_m(y_{i,n})=g_i(\sigma^m(t)(\pi^{(n)}))$,
the left hand side is equal to 
$\sigma^n(t)(\sigma^m(t)(\pi^{(n)}))=\sigma^{n+m}(t)(\pi^{(n)})
=\sigma^m(t)(\pi)$ too.
\end{proof}

Define a frame 
\[
\BBB^{\nr}=(\FS^{\nr},E\FS^{\nr},
\FS^{\nr}/E\FS^{\nr},\sigma,\sigma_1)
\] 
with $\sigma_1(Ex)=\sigma(x)$ for $x\in\FS^{\nr}$.

\begin{lemma}
The element $\uu'=\ff_1(\varkappa^{\nr}(E))\in\hat\WW(\tilde R)$ 
is a unit, and the ring homomorphism 
$\varkappa^{\nr}:\FS^{\nr}\to\hat\WW(\tilde R)$ 
is a $\uu'$-homomorphism of frames 
$\varkappa^{\nr}:\BBB^{\nr}\to\hat\DDD_{\tilde R}$.
\end{lemma}

\begin{proof}
Clearly $\varkappa^{\nr}$ commutes with the projections
to $\hat{\bar R}$ and with the Frobenius.
For $x\in\FS^{\nr}$ we compute
$\ff_1(\varkappa^{\nr}(Ex)=\ff_1(\varkappa^{\nr}(E))\cdot f(\varkappa^{\nr}(x))
=\uu'\cdot\varkappa^{\nr}(\sigma_1(Ex))$ as required.
It remains to show that $\uu'$ is a unit. 
The projection $\tilde R\to\bar k$ induces a local
homomorphism of local rings 
$\hat\WW(\tilde R)\to W(\bar k)$ that commutes with
$f$ and $\ff_1$.
The composition $\FS\to\FS^{\nr}\to\hat\WW(\tilde R)\to W(\bar k)$
commutes with Frobenius and is thus equal to
the homomorphism $t\mapsto 0$.
Thus $E$ maps to $p$ in $W(\bar k)$, so $\uu'$ maps to
$\ff_1(p)=v^{-1}(p)=1$ in $W(\bar k)$, and it follows
that $\uu'$ is a unit.
\end{proof}

From now on we assume that the hypotheses of 
Lemma \ref{Le-kappa-nr} are satisfied. Then
$\uu'$ is the image of $\uu$, and we get a
commutative square of frames,
where the horizontal arrows are $\uu$-homo\-morphisms
and the vertical arrows are strict:
\[
\xymatrix@M+0.2em@C+1em{
\BBB \ar[r]^-\varkappa \ar[d] & \DDD_R \ar[d] \\
\BBB^{\nr} \ar[r]^-{\varkappa^{\nr}} & \hat\DDD_{\tilde R}
}
\]
Here $\GGG_K$ acts on $\hat\DDD_{\tilde R}$ and
$\GGG_{K_\infty}$ acts on $\BBB^{\nr}$, and
$\varkappa^{\nr}$ is $\GGG_{K_{\infty}}$-equivariant.

\subsection{Identification of modules of invariants}

Now we can state the main result of this section.
Let $(M,\phi)$ be a Breuil window relative to $\FS\to R$
with associated $\BBB$-window $\PPP$, and let $\PPP^{\nr}$
be the base change of $\PPP$ to $\BBB^{\nr}$.
By definition we have $T^{\nr}(M,\phi)=T(\PPP^{\nr})$
as $\GGG_{K_\infty}$-modules.
Let $\PPP_\DDD$ be the base change of $\PPP$ to $\DDD_R$
and let $\hat\PPP_{\!\hat\DDD}$ be the
common base change of $\PPP^{\nr}$ and $\PPP_{\DDD}$ to
$\hat\DDD_{\tilde R}$. 
As in Section \ref{Se-frames}, multiplication by $\cc$
induces a $\GGG_{K_\infty}$-invariant homomorphism 
\[
\tau(\PPP^{\nr}):
T(\PPP^{\nr})\to T(\hat\PPP_{\!\hat\DDD}).
\]
We recall that the $\GGG_K$-module $T(\hat\PPP_{\!\hat\DDD})$ is
isomorphic to the Tate module of the 
$p$-divisible group associated to $(M,\phi)$;
see Proposition \ref{Pr-Tate-perfect}.

\begin{prop}
\label{Pr-bijective}
The homomorphism $\tau(\PPP^{\nr})$ is bijective.
\end{prop}

\begin{proof}
Let $h$ be the $\FS$-rank of $M$.
The source and target of $\tau(\PPP^{\nr})$
are free $\ZZ_p$-modules of rank $h$ which
are exact functors of $\PPP$; 
this is true for $T(\PPP^{\nr})$ 
since \eqref{Eq-B1} and \eqref{Eq-A2} are bijective,
and for $T(\hat\DDD_{\tilde R})$
by Proposition \ref{Pr-Tate-perfect}.

Consider first the case where the $p$-divisible
group associated to $\PPP$ is \'etale, which means
that $\PPP=(P,Q,F,F_1)$ has $P=Q$, and
$F_1:Q\to P$ is a $\sigma$-linear isomorphism.
Then a $\ZZ_p$-basis of $T(\PPP^{\nr})$ is
an $\FS^{\nr}$-basis of $P^{\nr}$, and a
$\ZZ_p$-basis of $T(\hat\PPP_{\hat\DDD})$
is a $\hat\WW(\tilde R)$-basis of $\hat P_{\tilde R}$. 
Since $\ZZ_p\to\hat\WW(\tilde R)$ is a local homomorphism
it follows that $\tau(\PPP^{\nr})$ is bijective.

Consider next the case $\PPP=\BBB$, which corresponds to the
$p$-divisible group $\mu_{p^\infty}$. 
Assume that the proposition does not hold for $\BBB$, 
i.e.\ that $\tau(\BBB^{\nr})$ is divisible by $p$. 
We may replace $k$ be an arbitrary perfect extension
since this does not change $\tau(\BBB)$, in particular
we may assume that $k$ is uncountable.
Let $\PPP_0$ be the etale
$\BBB$-window that corresponds to $\QQ_p/\ZZ_p$.
We consider extensions of $\BBB$-windows
$0\to\BBB\to\PPP_1\to\PPP_0\to 0$,
which correspond to extensions in 
$\Ext^1(\QQ_p/\ZZ_p,\mu_{p^\infty})$.
The image of $\tau(\PPP_1^{\nr})$ provides a
splitting of the reduction modulo $p$ of the 
exact sequence 
\[
0\to T(\hat\DDD_{\tilde R})\to T((\hat\PPP_1)_{\hat\DDD})
\to T((\hat\PPP_0)_{\hat\DDD})\to 0
\]
and thus the natural homomorphism
\begin{equation}
\label{Eq-zero}
\Ext^1_R(\QQ_p/\ZZ_p,\mu_{p^\infty})
\to\Ext^1_K(\ZZ/p\ZZ,\mu_p)\to
\Ext^1_{K_\infty}(\ZZ/p\ZZ,\mu_p)
\end{equation}
is zero. Now the first arrow in \eqref{Eq-zero} can be identified
with the obvious homomorphism of multiplicative groups
$1+\Fm_R\to K^*/(K^*)^p$; see \cite[Lemma 7.2]{Lau-Tate}
and its proof.
By our assumption on $k$ its image is uncountable.
Since for a finite extension $K'/K$ the homomorphism
$H^1(K,\mu_p)\to H^1(K',\mu_p)$ has finite kernel,
the kernel of the second map in \eqref{Eq-zero}
is countable. Thus the composition \eqref{Eq-zero}
cannot be zero, and the proposition is proved for $\PPP=\BBB$.

Finally let $\PPP$ be arbitrary.
Duality gives the following commutative diagram;
see Section \ref{Se-frames}.
\[
\xymatrix@M+0.2em{
T(\PPP^{\nr})\times T(\PPP^{t\nr}) 
\ar[r] \ar[d]_{\tau(\PPP^{\nr})\times\tau(\PPP^{t\nr})} &
T(\BBB^{\nr}) \ar[d]^{\tau(\BBB^{\nr})} \\
T(\hat\PPP_{\!\hat\DDD})\times T(\hat\PPP^t_{\!\hat\DDD}) 
\ar[r] &
T(\hat\DDD_{\tilde R})
}
\]
Since \eqref{Eq-B1} and \eqref{Eq-A2} are bijective,
the upper line of the diagram is a perfect bilinear 
form of free $\ZZ_p$-modules of rank $h$.
Proposition \ref{Pr-Tate-perfect} implies that 
the lower line is a bilinear form of free 
$\ZZ_p$-modules of rank $h$. 
We have seen that $\tau(\BBB^{\nr})$ is bijective.
These properties imply that $\tau(\PPP^{\nr})$ is bijective.
\end{proof}



\begin{thebibliography}{BBM}

\bibitem[Ar]{Artin-Etale-Coverings}
M.~Artin:
Etale coverings of schemes over Hensel rings. 
Amer.\ J.\ Math.\ {\bf 88} (1966), 915--934


\bibitem[BBM]{BBM} 
P.~Berthelot, L.~Breen, W.~Messing: 
Th\'{e}orie de Dieudonn\'{e} cristalline II.
Lecture Notes in Math.\ {\bf 930}, Springer Verlag, 1982 

\bibitem[Br]{Breuil}
C.~Breuil:
Schemas en groupes et corps des normes,
unpublished manuscript (1998)

\bibitem[El]{Elkik}
R.~Elkik:
Solutions d'\'{e}quations \`{a} coefficients dans 
un anneau hens\'{e}lien. 
Ann.\ Sci.\ \'{E}cole Norm.\ Sup.\ (4) {\bf 6} (1973), 553--603 

\bibitem[Fa]{Faltings-Integral}
G.~Faltings:
Integral crystalline cohomology over very ramified
valuation rings. JAMS {\bf 12} (1999) 117--144

\bibitem[Fo]{Fontaine-90}
J.-M.\ Fontaine:
Repr\'esentations $p$-adiques des corps locaux.
Grothendieck Festschrift II, 249--309,
Prog.\ Math.\ 87, Birkh\"auser, Boston, 1990

\bibitem[Gre]{Greco-Excellent}
S.~Greco:
Two theorems on excellent rings.
Nagoya Math.\ J.\ {\bf 60} (1976), 139--149

\bibitem[K]{Kim}
W.~Kim: The classification of $p$-divisible groups over
$2$-adic discrete valuation rings.
arxiv:1007.1904

\bibitem[Ki1]{Kisin-crys}
M.~Kisin: Crystalline representations and $F$-crystals,
{\em Algebraic geometry and number theory}, 459--496,
Progr.\ Math., Vol.\ 253, Birkh\"auser, 2006

\bibitem[Ki2]{Kisin-2adic}
M.~Kisin: Modularity of 2-adic Barsotti-Tate representations. 
Invent.\ Math.\ {\bf 178} (2009), 587--634

\bibitem[La1]{Lau-Tate}
E.~Lau:
Tate modules of universal $p$-divisible groups.
Compos.\ Math.\ {\bf 146} (2010), 220--232


\bibitem[La2]{Lau-Frames}
E.~Lau:
Frames and finite group schemes over complete
regular local rings.
Doc.\ Math.\ {\bf 15} (2010), 545--569

\bibitem[La3]{Lau-Relation}
E.~Lau:
Relations between crystalline Dieudonn\'e theory
and Dieudonn\'e displays. 
arXiv:1006.2720

\bibitem[Me]{Messing-Disp}
W.~Messing: 
Travaux de Zink.
S\'{e}minaire Bourbaki 2005/2006, exp.\ 964, 
Ast\'{e}risque {\bf 311} (2007), 341--364

\bibitem[Ra]{Raynaud-Henseliens}
M.~Raynaud:
Anneaux Locaux Hens\'eliens.
Springer LNM {\bf 169} (1970)

\bibitem[Se]{Seydi-Weierstrass}
H.~Seydi:
Sur la th\'{e}orie des anneaux de Weierstrass I. 
Bull.\ Sci.\ Math.\ (2) {\bf 95} (1971), 227--235

\bibitem[Ta]{Tate}
J.~Tate:
$p$-divisible groups.
Proceedings of a conference on local fields,
Dribergen (1966), 158--183, Springer Verlag, 1976

\bibitem[VZ]{Vasiu-Zink}
A.~Vasiu, Th.~Zink:
Breuil's classification of $p$-divisible groups over regular 
local rings of arbitrary dimension.
arXiv:0808.2792. To appear in:
Advanced Studies in Pure Mathematics, Proceeding of Algebraic and Arithmetic Structures of Moduli Spaces, Hokkaido University, Sapporo, Japan, September 2007

\bibitem[Zi1]{Zink-Disp}
Th.~Zink:
The display of a formal $p$-divisible group. 
Ast\'{e}risque {\bf 278} (2002), 127--248

\bibitem[Zi2]{Zink-DDisp}
Th.~Zink: 
A Dieudonn\'{e} theory for $p$-divisible groups.
Class field theory---its centenary and prospect, 
139--160, Adv.\ Stud.\ Pure Math.\ {\bf 30}, 
Math.\ Soc.\ Japan 2001

\end{thebibliography}
\end{document}